\documentclass[11pt]{article}

\usepackage{amsthm,amsmath,amssymb,amsbsy,amsfonts,graphics,epsfig}
\newcounter{geqncount} %
{\setcounter{equation}{\value{geqncount}}}

\newtheorem{theo}{Theorem}
\newtheorem{coro}{Corollary}
\newtheorem{lemma}{Lemma}

\newcommand{\beq}{\begin{equation}} \newcommand{\eeq}{\end{equation}}
\newcommand{\beqr}{\begin{eqnarray}}
\newcommand{\eeqr}{\end{eqnarray}}
\newcommand{\beqrn}{\begin{eqnarray*}}
\newcommand{\eeqrn}{\end{eqnarray*}} \newcommand{\brr}{\begin{array}}
\newcommand{\err}{\end{array}}
\newcommand{\bseq}{\begin{subequations}}
\newcommand{\eseq}{\end{subequations}}
\newcommand{\bef}{\begin{figure}} \newcommand{\eef}{\end{figure}}
\newcommand{\bec}{\begin{center}} \newcommand{\eec}{\end{center}}

\title{Interlacing and asymptotic properties of Stieltjes polynomials}

\author{A. Bourget and T. McMillen \\[.1in]
Department of Mathematics,
California State University at Fullerton \\
McCarthy Hall 154, Fullerton, CA  92834 \\
{\small {\tt abourget@fullerton.edu, tmcmillen@fullerton.edu}}
}

\begin{document}

\maketitle

\bibliographystyle{plain}


\begin{abstract}

\noindent Polynomial solutions to the generalized Lam\'e equation, the \textit{Stieltjes polynomials}, and the associated
\textit{Van Vleck polynomials} have been studied since the 1830's, beginning with Lam\'e in his studies of the Laplace equation on an ellipsoid, and in an ever widening variety of applications since.  
In this paper we show how the zeros of Stieltjes polynomials are distributed and 
present two new interlacing theorems.  We arrange the Stieltjes polynomials according to their Van Vleck zeros and show,  firstly, that the zeros of successive Stieltjes polynomials of the same degree interlace, and secondly, that the zeros of Stieltjes polynomials of successive degrees interlace.  We use these results to deduce new asymptotic properties of Stieltjes and Van Vleck polynomials.  We also show that no sequence of Stieltjes polynomials is orthogonal.


\end{abstract}

\section{Introduction and main results}
\label{s1}

Let $\alpha_1< \cdots <\alpha_p$ be any $p$ distinct real numbers, and let $\rho_1,\dots ,\rho_p$ be positive
numbers. The \textit{generalized Lam\'e equation}\footnote{Usually the term generalized Lam\'e equation refers to such equations in which the $\alpha_j$'s are allowed to be complex.  But here we only consider the real case.} is the second order ODE given by
\begin{equation} 
\label{lame1}
  S''(x) +  \sum_{j=1}^p  \frac{\rho_j}{x-\alpha_j}
  \, S'(x) = \frac{V(x)}{A(x)} \, S(x)
\end{equation}
where $A(x)=\prod_{j=1}^p(x-\alpha_j)$ and $V(x)$ is a polynomial of degree $p-2$. A result of Stieltjes \cite{MR1554669}, known as the Heine-Stieiltjes Theorem \cite{sz75}, says that there
exist exactly $k+p-2\choose k$ polynomials $V(x)$  for which \eqref{lame1} has a polynomial solution
$S$ of degree $k$.  These polynomial solutions are often called \textit{Stieltjes} or \textit{Heine-Stieltjes polynomials}, and the
corresponding polynomials $V(x)$ are known as \textit{Van Vleck polynomials}.  

In this paper we will consider the case when $p=3$, in which case \eqref{lame1} is a Heun equation\footnote{The Heun equation is \eqref{lame1} with $p=3$ and where the $\rho_j$'s are allowed to be negative but must satisfy some other conditions (cf. \cite{MR1928260}).}, and the Heine-Stieltjes Theorem says that there are
 $k+1$ values of $\nu$ for which \eqref{lame1} has a polynomial solution of degree $k$ with $V(x)=\mu(x-\nu)$.  We refer to these values of $\nu$ as   the Van Vleck zeros of order $k$.

The equation \eqref{lame1} was studied by Lam\'e in the 1830's in the special case, $\rho_j=1/2$,
$\alpha_1+\alpha_2+\alpha_3 = 0$ in connection with the separation of variables in the Laplace equation using
elliptical coordinates \cite[Ch. 23]{whwa96}. The equation has since found a strikingly wide variety of other applications, from electrostatics \cite{MR1662694, MR2345246, dias00, mcboag09} to completely quantum integrable systems such as the quantum C. Neumann oscillators, the asymmetric top and the geodesic flow on an ellipsoid \cite{agbo07, grve08, bomc09, hawi95}.  In particular, the zeros of the Stieltjes polynomials may be nicely interpreted as the equilibrium positions of $k$ unit charges in a logarithmic potential in which at each position $\alpha_j$ in the complex plane is fixed a charge of magnitude $\rho_j$.

Much is known about the properties of Stieltjes and Van Vleck polynomials for a fixed degree of the  Stieltjes
polynomial (see, e.g. \cite{vo04} for recent results), but there are few results relating  Van Vleck and Stieltjes zeros of different degrees.  
A few important facts are that the zeros of any
Stieltjes polynomial are simple, lie inside the interval $(\alpha_1,\alpha_3)$, and none of them can equal
$\alpha_2$ or its corresponding Van Vleck zero.  Similarly, for fixed $k$, the Van Vleck zeros  are distinct and
also lie within $(\alpha_1,\alpha_3)$.  The proofs of these results can be found in  \cite{vv98, sz75, MR0230954}.

Recently we  showed in \cite{bomcva08} that the Van Vleck zeros of successive orders interlace.  That
is, if the Van Vleck zeros of order $k$ are written in increasing order as
 $\nu_1^{(k)} < \nu_2^{(k)} < \cdots < \nu_{k+1}^{(k)}$, then
\begin{equation}
\alpha_1 < 
\nu_{1}^{(k+1)} < \nu_1^{(k)} < \nu_{2}^{(k+1)} < \nu_2^{(k)} < \cdots < \nu_{k+1}^{(k)} < \nu_{k+2}^{(k+1)} < \alpha_3.
\label{inequality}
\end{equation}

Much of the research in the past several years has focused on the asymptotic properties of the zeros of Stieltjes and Van Vleck polynomials as the degree of the corresponding Stieltjes polynomials
tends toward infinity \cite{bosh08, bo08, MR1928260, mamaor05}.   Interlacing theorems such as the one above are interesting not only for what they tell us about the zeros for a finite degree, but they also help us to understand such asymptotic limits.  They are ``classical'' results in the sense of being statements about finite degree polynomials, and they are also a bridge to connect other classical results with asymptotic limits.  Our results below are further steps in this direction.

In this paper we present a theorem on the distribution of the Stieltjes zeros and two additional interlacing theorems.  For each positive integer $k$ we label the $k+1$ Stieltjes polynomials of degree $k$  according to their Van Vleck zeros, as $S_j^{(k)}(x)$.  That is, $S_j^{(k)}(x)$ is the polynomial of degree $k$ that satisfies
\begin{equation} \label{lame2}
 \left[\frac{d^2}{dx^2} +  \sum_{j=1}^3  \frac{\rho_j}{x-\alpha_j}
  \, \frac{d}{dx} 
  - \frac{\mu\left(x-\nu_j^{(k)}\right)}{A(x)} \right] S_j^{(k)}(x) = 0
\end{equation}

Our first result is a strengthening of a classical result of Stieltjes:
\begin{theo}
\label{lem1}
There are exactly $j-1$ zeros of $S_j^{(k)}$ in the interval $(\alpha_1,\alpha_2)$, and $k-j+1$ zeros of $S_j^{(k)}$ in the interval $(\alpha_2,\alpha_3)$.  Moreover, there are no zeros of $S_j^{(k)}$ between $\alpha_2$ and its corresponding Van Vleck zero $\nu_j^{(k)}$.
\end{theo}

Next we show that the zeros of successive Stieltjes polynomials of the same degree interlace.  In the third theorem we establish that the zeros of the $j$th Stieltjes polynomials of successive degrees interlace.  The interlacing properties are illustrated in Figure~\ref{figure1}.
\begin{theo}
\label{theo1}
The zeros of $S_j^{(k)}$ and $S_{j+1}^{(k)}$ interlace; between any two consecutive zeros of $S_j^{(k)}$ there is exactly one zero of $S_{j+1}^{(k)}$.  Moreover, the smallest zero of $S_{j+1}^{(k)}$ is less than the smallest zero of $S_j^{(k)}$.
That is, let $x_{i,j}^{(k)}$ be the zeros of the Stieltjes polynomial $S_j^{(k)}$, arranged in increasing order.    Then
\begin{equation}
\alpha_1 < x_{1,j+1}^{(k)} < x_{1,j}^{(k)} < x_{2,j+1}^{(k)} < x_{2,j}^{(k)} < \cdots < x_{k,j+1}^{(k)} < x_{k,j}^{(k)} < \alpha_3.
\end{equation}  
\end{theo}

\begin{theo}
\label{theo2}
The zeros of $S_j^{(k)}$ and $S_i^{(k+1)}$ interlace if and only if $i=j$ or $i=j+1$.
 That is, if $i=j$ or $i=j+1$ then
\begin{equation}
\alpha_1 < x_{1,i}^{(k+1)} < x_{1,j}^{(k)} < x_{2,i}^{(k+1)} < x_{2,j}^{(k)} < \cdots < x_{k,j}^{(k)} < x_{k+1,i}^{(k+1)} < \alpha_3.
\end{equation}  
Otherwise the zeros of $S_j^{(k)}$ and $S_i^{(k+1)}$ do not interlace.
\end{theo}

\begin{figure}[h]
\begin{center}
        \includegraphics*[scale=.75]{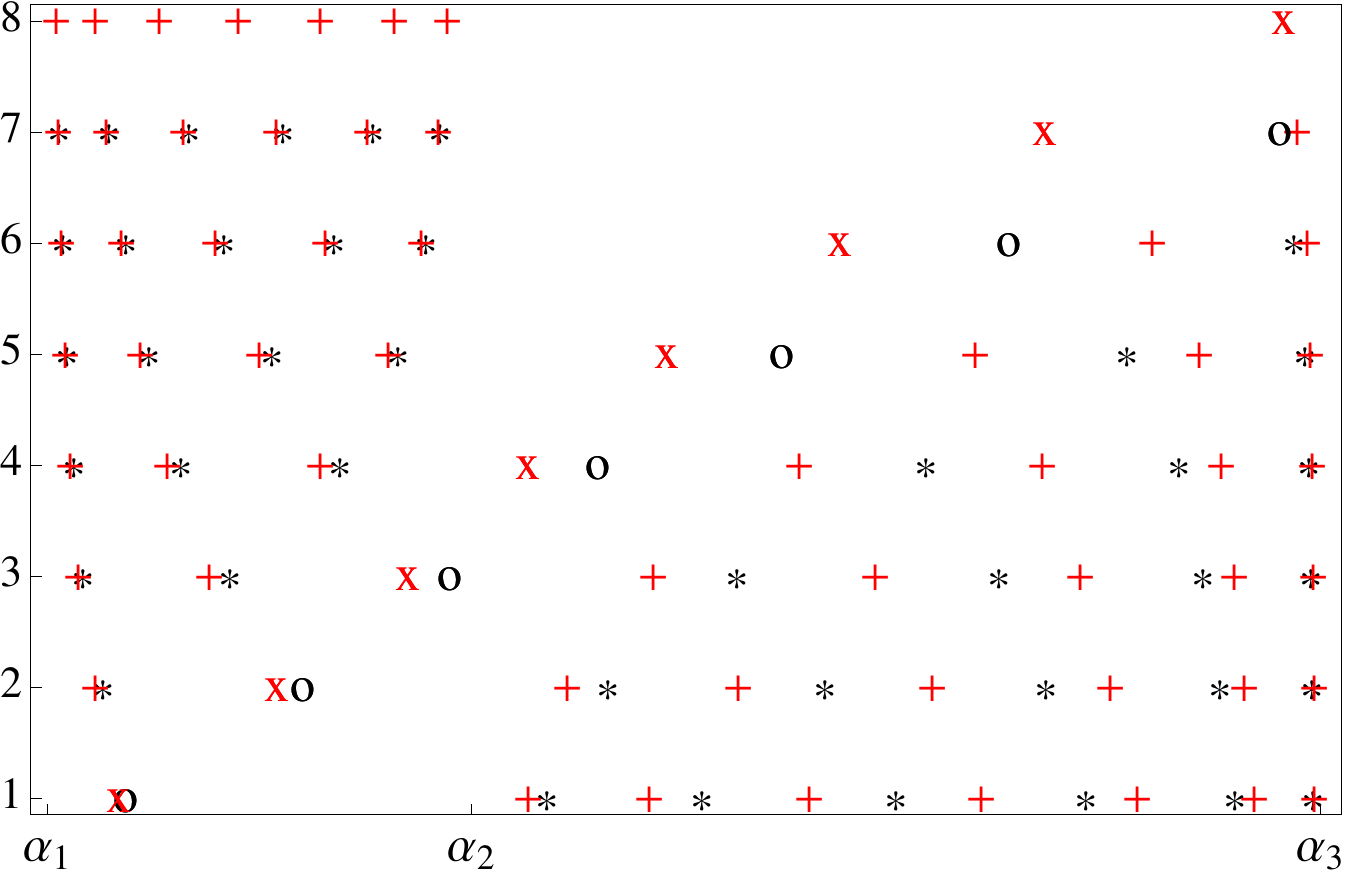}
\end{center}
\caption{Zeros of Stieltjes and Van Vleck polynomials.  On row $j$ are shown
$*$:  zeros of $S_j^{(k)}$; \ $+$: zeros of $S_j^{(k+1)}$; \ O:  Van Vleck zero $\nu_j^{(k)}$; \ X:  Van Vleck zero $\nu_j^{(k+1)}$,  for $k=6$.  Values of the parameters are $(\alpha_1, \alpha_2, \alpha_3) = (-1, 0, 2), \; (\rho_1, \rho_2, \rho_3) = (1, 2, 1/3)$.}
\label{figure1}
\end{figure}

In the following section we prove these results.  Following this, in Section \ref{s3}, we show how these results can be combined with known asymptotic properties of Stieltjes and Van Vleck polynomials to produce new results.  In particular, we construct sequences of Van Vleck zeros to converge to any number in $[\alpha_1, \alpha_3]$, and calculate the asymptotic zero distribution of the Stieltjes polynomials associated with these Van Vleck zeros.

Here we also take up the question of whether or not there exist any orthogonal sequences of Stieltjes polynomials.  This is a natural question to ask, for several reasons.
For one thing, we see from Theorem~\ref{theo2}  that any sequence of Stieltjes polynomials $\{S_{j_k}^{(k)}\}$ with $j_{k+1}=j_k$ or $j_k+1$ shares at least one thing in common with orthogonal polynomials, namely the well known fact that the zeros of orthogonal polynomials of successive degrees interlace.
Moreover,  equation \eqref{lame1} with $p=2$ is the Jacobi differential equation, and the polynomial solutions are the Jacobi polynomials which are, of course, orthogonal \cite{to05}.  We will also show that certain of the sequences $\{S_{j_k}^{(k)}\}$ have the same asymptotic zero distribution as sequences of orthogonal polynomials.  Additionally, Ismail \cite{is05} has shown that under fairly general conditions, equilibrium positions of particles in a logarithmic potential are the zeros of orthogonal polynomials.  In our setting, the sufficient condition is only violated by the fact that the term in front of $S'(x)$ in \eqref{lame1} is not differentiable at $\alpha_2$.  Ismail's theorem gives only sufficient conditions for orthogonality, however, so the possibility of orthogonality of Stieltjes polynomials has remained an important open question. 
It is worth noting that, in spite of these intriguing similarities,  in the over 150 year  literature on the Lam\'e equation, 
we have found no definitive statement or conjecture concerning orthogonality of polynomial solutions of \eqref{lame1} when $p\geq 3$.  

We end Section~\ref{s3} by proving that there are no sequences of Stieltjes polynomials that are orthogonal with respect to a measure.
We conclude in  Section~\ref{s4} with a
 few remarks and comments on future directions.

\section{Proofs of main results }
\label{s2}

For the proofs of the theorems  we will need the following lemma, which is basically the Sturm separation theorem applied in the intervals $(\alpha_1, \alpha_2)$ and $(\alpha_2, \alpha_3)$.
\begin{lemma}
There is a zero of $S_{j+1}^{(k)}$  between every two zeros of $S_j^{(k)}$ in the interval $(\alpha_1, \alpha_2)$, and between every zero of $S_j^{(k)}$ in $(\alpha_1, \alpha_2)$ and either of the singular points $\alpha_1, \alpha_2$.
Likewise, there is a zero of $S_{j}^{(k)}$  between every two zeros of $S_{j+1}^{(k)}$ in the interval $(\alpha_2, \alpha_3)$, and between every zero of $S_{j+1}^{(k)}$ in $(\alpha_2, \alpha_3)$ and either of the singular points $\alpha_2, \alpha_3$.
\label{lem2}
\end{lemma}

\begin{proof}
First we note that the constant $\mu$ in \eqref{lame1} is determined by the degree $k$ of the Stieltjes polynomials by substitution and identification of the powers, as
\begin{equation}
\mu = \mu_k = k\left( k-1 + \rho_1+\rho_2+\rho_3\right). \label{mu}
\end{equation}
For the remainder of this proof, since we are dealing with Stieltjes polynomials of fixed degree $k$, we will omit the superscripts and simply write $S_j = S_j^{(k)}$ and $S_{j+1}=S_{j+1}^{(k)}$.
Now, the Stieltjes polynomials $S_j$ and $S_{j+1}$ satisfy
\begin{eqnarray}
 \left[\frac{d^2}{dx^2} +  \sum_{j=1}^3  \frac{\rho_j}{x-\alpha_j}
  \, \frac{d}{dx} 
  - \frac{\mu_k\left(x-\nu_j^{(k)}\right)}{A(x)} \right] S_j(x) = 0,
 \label{Sj}\\
 \left[\frac{d^2}{dx^2} +  \sum_{j=1}^3  \frac{\rho_j}{x-\alpha_j}
  \, \frac{d}{dx} 
  - \frac{\mu_k\left(x-\nu_{j+1}^{(k)}\right)}{A(x)} \right] S_{j+1}(x) = 0.
 \label{Sj+1}
\end{eqnarray}
Define the integrating factor
\beq 
J(x) = \prod_{j=1}^3 \left| x - \alpha_j \right|^{\rho_j}.
\eeq
Then 
\beq
J'(x) = J(x)\sum_{j=1}^3  \frac{\rho_j}{x-\alpha_j}.
\eeq  
Upon  multiplying \eqref{Sj} by $S_{j+1}$ and \eqref{Sj+1} by $S_j$, taking
the difference of the result, and then multiplying by $J$, we obtain
\begin{equation}
\frac{d}{dx} \left[ J \left(S_{j+1}' S_j - S_{j+1} S_j ' \right)\right]  = 
Q  S_j S_{j+1},
\label{sturm}
\end{equation}
where
\beq
Q(x) = \left(\nu_{j} -\nu_{j+1}\right) \frac{J(x)}{A(x)}.
\eeq
Note that $Q<0$ in $(\alpha_1, \alpha_2)$ and $Q>0$ in $(\alpha_2, \alpha_3)$.
Now, consider two consecutive zeros of $S_j$, $x_1$ and $x_2$, in the interval $(\alpha_1, \alpha_2)$. 
Then $S_j'$ must alternate signs at $x_1$ and $x_2$.  Thus, 
\beq
\left.J\left(S_{j+1}'S_j-S_{j+1}S_j'\right)\right|_{x=x_1}^{x=x_2} = 
-\left.J\left(S_{j+1}S_j'\right)\right|_{x=x_1}^{x=x_2}
\eeq
is positive if $S_jS_{j+1}>0$  in $(x_1, x_2)$, and negative if $S_jS_{j+1}<0$ in $(x_1, x_2)$.  But \eqref{sturm} implies that this expression is negative if $S_jS_{j+1} > 0$ in $(x_1, x_2)$ and positive if $S_jS_{j+1} < 0$ in $(x_1, x_2)$.  Thus it must be that $S_{j+1}$ changes sign in $(x_1, x_2)$, and we have shown that between any two consecutive zeros of $S_j$ in $(\alpha_1, \alpha_2)$, there is a zero of $S_{j+1}$.

Now let $x_1$ be the smallest zero  of $S_j$ in $(\alpha_1, \alpha_2)$.  Then, since $J(\alpha_1)=0$,
\beq
\left.J\left(S_{j+1}'S_j-S_{j+1}S_j'\right)\right|_{x=\alpha_1}^{x=x_1} = 
-J(x_1)S_{j+1}(x_1)S_j'(x_1)
\eeq
also has the same sign as $S_jS_{j+1}$ in $(\alpha_1, x_1)$ if $S_{j+1}$ does not change sign in this interval, again contradicting \eqref{sturm}.  A similar argument can be applied to the largest zero of $S_j$ in $(\alpha_1, \alpha_2)$ and $\alpha_2$, which establishes that between every zero of $S_j$ in $(\alpha_1, \alpha_2)$ and either of the singular points $\alpha_1, \alpha_2$, there is a zero of $S_{j+1}$.

This argument may be exactly repeated in the interval $(\alpha_2, \alpha_3)$, noting that $Q>0$ in this interval.
It follows that between any two consecutive zeros of $S_{j+1}$ in $(\alpha_2,\alpha_3)$, or between $\alpha_2$ and the smallest zero of $S_{j+1}$ in $(\alpha_2,\alpha_3)$, or between the largest zero of $S_{j+1}$ in $(\alpha_2, \alpha_3)$ and $\alpha_3$, 
there is a zero of $S_j$.  
\end{proof}

\begin{proof}[Proof of Theorem \ref{lem1}]
According to a classical result of Stieltjes \cite{MR1554669, whwa96}, every possible distribution of the zeros of the Stieltjes polynomials in the intervals $(\alpha_1, \alpha_2)$ and $(\alpha_2,\alpha_3)$ occurs.  That is, 
for each integer $k$ and any integer $0\leq m\leq k$, there is a Stieltjes polynomial of degree $k$ with $m$ zeros in $(\alpha_1, \alpha_2)$ and $k-m$ zeros in $(\alpha_2,\alpha_3)$.  Thus it suffices to show that there are at least as many zeros of $S_{j+1}^{(k)}$ than of $S_j^{(k)}$ in $(\alpha_1,\alpha_2)$.  But this is an immediate consequence of Lemma \ref{lem2}, so the first part of the theorem is proved.

For the second part of the theorem, according to Shah \cite[cf. Theorem 2]{MR0230954}, between any zero of $S(x)$ and $\nu$ there is either a zero of $S'(x)$ or $\alpha_2$.  Suppose $\nu>\alpha_2$, and that
 there is a zero $x_1$ of $S(x)$ between $\alpha_2$ and $\nu$.  Since there is a zero of $S'(x)$ between $x_1$ and $\nu$ there must be a zero $x_2$ of $S(x)$ greater than $\nu$.  We may assume that $x_1$ and $x_2$ are consecutive.  But since the zeros of $S(x)$ are simple there cannot be a zero of $S'(x)$ in both intervals $(x_1, \nu)$ and $(\nu, x_2)$, which is a contradiction.   
 The case when $\nu<\alpha_2$ is similar.
\end{proof}

\begin{proof}[Proof of Theorem \ref{theo1}]
Combining Lemma \ref{lem2}  with Theorem \ref{lem1}, we see that between the $j-1$ zeros of $S_j$ in $(\alpha_1, \alpha_2)$ there are $j-2$ zeros of $S_{j+1}$.  The other two zeros of $S_{j+1}$ in $(\alpha_1,\alpha_2)$ lie between $\alpha_1$ and the smallest zero of $S_j$ and between the largest zero of $S_j$ in $(\alpha_1, \alpha_2)$ and $\alpha_2$.  Similarly, between the $k-j$ zeros of $S_{j+1}$ in $(\alpha_2,\alpha_3)$ there are $k-j-1$ zeros of $S_j$.  And the other two zeros of $S_j$ in $(\alpha_2, \alpha_3)$ lie between $\alpha_2$ and the smallest zero of $S_{j+1}$ in $(\alpha_2, \alpha_3)$ and between the largest zero of $S_{j+1}$ in $(\alpha_2, \alpha_3)$ and $\alpha_3$.  We have thus accounted for all of the zeros of $S_j$ and $S_{j+1}$, and the interlacing is proved.
\end{proof}

We have actually proved a stronger statement than Theorem \ref{theo1}.  We note this in the following corollary.

\begin{coro}
Let $l>j$.  Then between any two zeros of $S_j^{(k)}$ in $(\alpha_1, \alpha_2)$ there is a zero of $S_l^{(k)}$.  Between any two zeros of $S_l^{(k)}$ in $(\alpha_2, \alpha_3)$ there is a zero of $S_j^{(k)}$.
\end{coro}

\begin{proof}[Proof of Theorem \ref{theo2}]
First we prove the ``only if'' part.  In order for the zeros of $S_i^{(k+1)}$ and $S_j^{(k)}$ to interlace, it must be the case that the smallest zero of $S_i^{(k+1)}$ is smaller than the smallest zero of $S_j^{(k)}$, which is impossible if $i<j$ since then 
 there would be fewer zeros of $S_i^{(k+1)}$ than of $S_j^{(k)}$ in the interval $(\alpha_1,\alpha_2)$.  Similarly, if $i>j+1$ then there are fewer zeros of $S_i^{(k+1)}$ than of $S_j^{(k)}$ in the interval $(\alpha_2, \alpha_3)$.

For the ``if'' part, 
as in the proof of Lemma \ref{lem2}, we derive the following expression
\begin{equation}
\frac{d}{dx} \left[ J \left(\frac{dS_i^{(k+1)}}{dx} S_j^{(k)} - S_i^{(k+1)} \frac{dS_j^{(k)}}{dx} \right)\right]  = QS_j^{(k)} S_i^{(k+1)},
\label{sturm2}
\end{equation}
where in this case
\begin{eqnarray}
Q(x) &=&
\left(\mu_{k+1}-\mu_k\right)\frac{J(x)}{A(x)}\left( x - \hat{\nu}_i\right), \quad
\mbox{and} \\
\hat{\nu}_i & = & \frac{\mu_{k+1}\nu_i^{(k+1)} - \mu_k\nu_j^{(k)}}{\mu_{k+1}-\mu_k}.
\end{eqnarray}

Note that \eqref{inequality} implies
\beq 
\hat{\nu}_j < \nu_j^{(k+1)}   < \alpha_3
\quad \mbox{and} \quad
\alpha_1 < \nu_{j+1}^{(k+1)} < \hat{\nu}_{j+1}.
\eeq 
Suppose, for the moment, that $\alpha_1 < \hat{\nu}_j < \alpha_2$.  Then, since $Q>0$ in $(\hat{\nu}_j, \alpha_2)$, between every two consecutive zeros of $S_j^{(k+1)}$ in $(\hat{\nu}_j,\alpha_2)$, and between the largest zero of $S_j^{(k+1)}$ in $(\hat{\nu}_j,\alpha_2)$ and $\alpha_2$, there is a zero of $S_j^{(k)}$.  Since $Q<0$ in $(\alpha_1, \hat{\nu}_j)$, there is a zero of $S_j^{(k+1)}$ between every two zeros of $S_j^{(k)}$ in $(\alpha_1, \hat{\nu}_j)$, and between $\alpha_1$ and the smallest zero of $S_j^{(k)}$ in $(\alpha_1, \hat{\nu}_j)$.  This accounts for all of the $j-1$ zeros of $S_j^{(k)}$ and of $S_j^{(k+1)}$ in $(\alpha_1, \alpha_2)$.  Similarly, since $Q<0$ in $(\alpha_2, \alpha_3)$, between every two zeros of $S_j^{(k)}$ in $(\alpha_2, \alpha_3)$ and between every zero of $S_j^{(k)}$ in $(\alpha_2, \alpha_3)$ and either of the singular points $\alpha_2, \alpha_3$, there is a zero of $S_j^{(k+1)}$.  This accounts for the $k$ zeros of $S_j^{(k)}$ and the $k+1$ zeros of $S_j^{(k+1)}$ and the interlacing is proved in this case.

A nearly identical argument holds in the case when $\alpha_2<\hat{\nu}_{j+1}<\alpha_3$ to show that the zeros of $S_{j+1}^{(k+1)}$ and $S_j^{(k)}$ interlace.
The proof is thus completed by the following lemma.
\end{proof}

\begin{lemma}
$\alpha_1<\hat{\nu}_j<\alpha_2$ and $\alpha_2 < \hat{\nu}_{j+1} < \alpha_3$. 
\end{lemma}
\begin{proof}
Suppose that $\hat{\nu}_j\leq\alpha_1$.  Then it would be the case that $Q>0$ in $(\alpha_1, \alpha_2)$ and by arguments similar to the above, between every two zeros of $S_j^{(k+1)}$ in $(\alpha_1, \alpha_2)$ and between every zero of $S_j^{(k+1)}$ in $(\alpha_1, \alpha_2)$ and either of the singular points $\alpha_1, \alpha_2$, there is a zero of $S_j^{(k)}$.  But this would imply the existence of at least $j$ zeros of $S_j^{(k)}$ in $(\alpha_1, \alpha_2)$, which contradicts Theorem \ref{lem1}.  Likewise, if $\hat{\nu}_j\geq\alpha_2$, an similar argument would imply the existence of at least $j$ zeros of $S_j^{(k+1)}$ in $(\alpha_1, \alpha_2)$, which is also a contradiction. This  argument may be repeated to prove the second statement. 
\end{proof}

\section{Asmyptotic properties}
\label{s3}

We now investigate how the zeros of Stieltjes polynomials distribute over the interval $(\alpha_1, \alpha_3)$ in the limit  as the degree of the Stieltjes polynomials tends to infinity.  
Recall that Theorem \ref{lem1} says that if $\nu<\alpha_2$, then there are $j-1$ zeros of $S_{j}^{(k)}$ in $(\alpha_1, \nu)$ and $k-j+1$ zeros in $(\alpha_2, \alpha_3)$, and similarly if $\nu\geq\alpha_2$, \textit{mutatis mutandi}.

In \cite{bosh08} it is shown that the Van Vleck zeros have the asymptotic distribution given by a probability measure
supported on $(\alpha_1, \alpha_3)$, with density $\rho_V(x)$ given by the following: 
\begin{equation}
\rho_V(x) = \left\{\begin{array}{ll}
\displaystyle \frac{1}{2\pi} \int_{\alpha_2}^{\alpha_3} \frac{ds}{\sqrt{A(s)(x-s)}}
& \; \mbox{ \ if \ } \; \alpha_1 < x < \alpha_2, \\[.2in]
\displaystyle \frac{1}{2\pi} \int_{\alpha_1}^{\alpha_2} 
\frac{ds}{\sqrt{A(s)(x-s)}}
& \; \mbox{ \ if \ } \; \alpha_2 < x < \alpha_3.
\end{array}\right.
\label{vvasy}
\end{equation}
Thus, given any point $\nu\in(\alpha_1, \alpha_3)$ there is a sequence of Van Vleck zeros $\left\{\nu_{j_k}^{(k)}\right\}$ that converges to $\nu$ as $k\rightarrow\infty$.  So we may ask, What is the distribution of the corresponding Stieltjes zeros in this same limit?  A partial answer to this question is given by Theorem \ref{lem1}, which implies that the limiting density is supported on a subset of  $(\alpha_1, \alpha_3)\setminus I$, where $I$ is the interval bounded by $\alpha_2$ and $\nu$.  The question can be  resolved by applying the results of \cite{MR1928260}.


Mart\'inez-Finkelshtein and Saff \cite{MR1928260} consider the more general case of \eqref{lame1} for an arbitrary number $p$ of $\alpha_i$'s.  The authors fix relative proportions $\theta_1, \dots, \theta_{p-1}$ in each of the intervals $(\alpha_i, \alpha_{i+1}), \; i=1, \dots, p-1$, and extract a sequence of Stieltjes polynomials such that as the degree of each polynomial in the sequence tends to infinity, the fraction of its zeros in the interval $(\alpha_i, \alpha_{i+1})$ tends to $\theta_i$.  Under this assumption the authors
derive asymptotic results for the zeros of Stieltjes and Van Vleck polynomials.  We specialize as before to the $p=3$ case and recast these calculations in the light of our results.

Let $0\leq \theta\leq 1$.  Theorem 1 of  \cite{MR1928260} gives the asymptotic distribution of the zeros of a sequence  of Stieltjes polynomials. We interpret these in our setting in the following way.  Suppose the sequence $\left\{S_{j_k}^{(k)}\right\}$ is chosen such that the fraction of the zeros of $S_{j_k}^{(k)}$ in $(\alpha_1, \alpha_2)$ tends to $\theta$.    Some simple but tedious calculations shows that the limit $\lim_{k\rightarrow\infty}\nu_{j_k}^{(k)} = \nu$ of the corresponding Van Vleck zeros is determined by
\beq
\frac{1}{\pi}\int_{\alpha_1}^{\min(\alpha_2, \nu)} \sqrt{\left| \frac{\nu - x}{A(x)}\right|} dx = \theta.
\label{nu}
\eeq
Moreover, let $I$ by the interval bounded by $\alpha_2$ and $\nu$.  Then the asymptotic distribution of the Stieltjes polynomials $S_{j_k}^{(k)}$ is given by
\beq
\rho_S(x) = \left\{ \begin{array}{ll}
\displaystyle \frac{1}{\pi}\sqrt{ \frac{\nu-x}{A(x)}} & \; \mbox{ \ if \ } \; x \in (\alpha_1, \alpha_3)\setminus I \\[.1in]
0 & \; \mbox{ \ if \ } \; x \in I
\end{array} \right.
\label{Sden}
\eeq
We note, in particular, that if $\nu=\alpha_1, \alpha_2$ or $\alpha_3$ then $\rho_S$ is the so-called ``arcsin distribution'' supported on the intervals $(\alpha_2, \alpha_3)$, $(\alpha_1, \alpha_3)$ and $(\alpha_1,\alpha_2)$, respectively.


Now, by Theorem \ref{lem1}, $j_k-1$ zeros of $S_{j_k}^{(k)}$ lie in $(\alpha_1, \alpha_2)$.  So, if $j_k/k\rightarrow\theta$, then the fraction of zeros of $S_{j_k}^{(k)}$ in $(\alpha_1, \alpha_2)$ tends to $\theta$.  (One such sequence, that also satisfies $j_{k+1}=j_k$ or $j_k+1$, is $j_k=\lceil k\theta+1\rceil $.)  Combining our results with the asymptotic properties, we have the following.
\begin{theo}
(i)  Let $\{j_k\}$ be a sequence of positive integers such that $j_k/k \rightarrow \theta$.  Then we have the limit of sequences of Van Vleck zeros:
\beq
\lim_{k\rightarrow\infty}\nu_{j_k}^{(k)}=\nu, 
\eeq
where $\nu$ is determined by \eqref{nu}.   
(ii)  Given any $\nu\in[\alpha_1, \alpha_3]$, if $\theta$ is defined by \eqref{nu} and $j_k/k\rightarrow\theta$, then the sequence $\left\{\nu_{j_k}^{(k)}\right\}$ converges to $\nu$.
(iii)  If the sequence $\{j_k\}$ satisfies $j_{k+1} = j_k$ or $j_k+1$, then the sequence of Stieltjes polynomials $\left\{ S_{j_k}^{(k)}\right\}$ is an infinite sequence of polynomials with interlacing zeros and asymptotic zero distribution given by \eqref{Sden}.
\end{theo}

We may also calculate conditions under which the limiting distribution of Stieltjes zeros is supported throughout $(\alpha_1, \alpha_3)$.  A necessary condition  given by Theorem \ref{lem1} is that $\nu =\alpha_2$, and \eqref{Sden} tells us that this is also a sufficient condition.  
In order for the limiting distribution to be supported on $(\alpha_1, \alpha_3)$, the fraction of the zeros of $S_{j_k}^{(k)}$ in $(\alpha_1, \alpha_2)$ must tend to  $\theta_c$, where $\theta_c$ is calculated by setting $\nu=\alpha_2$ in \eqref{nu}.  A simple calculation shows that this is (see also \cite[Prop. 1]{MR1928260}):
\beq
\theta_c = \frac{2}{\pi} \sin^{-1}\sqrt{\frac{\alpha_2 - \alpha_1}{\alpha_3 - \alpha_1}}.
\eeq

One remaining question is whether there exist sequences of orthogonal Stieltjes polynomials.  
Orthogonality is a rather strict condition on sequences of polynomials, but, as we mentioned in the Introduction, there are compelling reasons to suspect that some sequences of Stieltjes polynomials might be orthogonal.  
It is well known (see, e.g. \cite{sz75}) that if a sequence $\{p_n\}$ of monic polynomials of degree $n$ is orthogonal with respect to some measure, then the zeros of $p_n$ and $p_{n+1}$ interlace.  Moreover, $\{p_n\}$ is orthogonal if and only if there exist sequences $\{a_n\}$ and $\{b_n\}$, with $a_n\in\mathbb{R}$ and $b_n>0$ such that
$p_n= (x-a_n)p_{n-1} - b_np_{n-2}.
$
We have found no simple contradiction to emerge from assuming that this holds for sequences of Stieltjes polynomials.
Additionally, if $a_n\rightarrow a\in\mathbb{R}$ and $b_n\rightarrow b\in (0,\infty)$, then according to Theorem 5.3 of \cite{ne79}, the polynomials $p_n$ have the asymptotic zero distribution $\omega_{[\alpha, \beta]}$ with density supported on the interval $[\alpha, \beta]$ given by
\beq
\frac{d\omega_{[\alpha, \beta]}(x)}{dx} =
\frac{1}{\pi\sqrt{(\beta-x)(x-\alpha)}},
\label{opden}
\eeq
where $\alpha=a-2/b$ and $\beta=a+2/b$.

Interestingly, the density \eqref{opden} is identical to \eqref{Sden} when $\nu$ in \eqref{Sden} is one of $\alpha_1, \alpha_2, \alpha_3$.  Thus, if $\theta = 0, 1$ or $\theta_c$ the sequence 
 $\left\{ S_{\lceil k\theta + 1\rceil}^{(k)}\right\}$ is an infinite sequence of polynomials with interlacing zeros  and asymptotic zero distribution supported in an open interval and given by a density of the form \eqref{opden}.  These sequences thus share many properties of orthogonal polynomial sequences.  However, despite these commonalities, we have the following result.
 
 \begin{theo}
 If $p=3$, then given any measure $\omega$ supported on $(\alpha_1, \alpha_3)$, there is no sequence of Stieltjes polynomials orthogonal with respect to $\omega$.
 \end{theo}

\begin{proof} 
Our argument is based on the uniqueness of the orthogonal measure for products of
Stieltjes polynomials.  Let $\{S^{(n)}\}_{n=1}^{\infty}$ be a sequence of Stieltjes polynomials.  In \cite{vo99}, Volkmer shows that for $n \neq m$, the two variables polynomials
$\prod_{k=1}^2 S^{(n)}(x_k)$ and $\prod_{k=1}^2 S^{(m)}(x_k)$ are orthogonal with respect to the measure
\beq 
\mu(x_1,x_2)= \prod_{j=1}^3 \prod_{k=1}^2 |x_k-\alpha_j|^{\rho_j-1} (x_2-x_1) 
\eeq
on the open rectangle $(\alpha_1,\alpha_2) \times (\alpha_2,\alpha_3)$. 
The measure can easily be extended to the bigger rectangle $R=(\alpha_1,\alpha_3) \times (\alpha_1,\alpha_3)$ by letting
\beq
\nu(x_1, x_2) = \chi_{[\alpha_1, \alpha_2]}(x_1) \chi_{[\alpha_2, \alpha_3]}(x_2) \mu( x_1, x_2),
\eeq
where $\chi_E(x)$ is the characteristic function of $E$.

Since the condition
\[ \int_{R} e^{x_1^2+x_2^2}  \ d\nu(x_1,x_2) < \infty \]
holds, Theorem 3.1.17 in \cite{duxu01} implies that $\nu$ is deterministic, i.e. $\nu$ is uniquely determined by its
moments. It then follows, that $\nu$ is the unique measure on $R$ for which $\prod_{k=1}^2 S^{(n)}(x_k)$ and
$\prod_{k=1}^2 S^{(m)}(x_k)$ are orthogonal.

It is now easy to establish that Stieltjes polynomials are not orthogonal on $(\alpha_1,\alpha_3)$ with respect
to any measure. Assume the contrary, i.e. assume there exists a measure $\omega$ with support in
$(\alpha_1,\alpha_3)$ such that
\beq 
\int_{\alpha_1}^{\alpha_3} S^{(n)}(x) S^{(m)}(x) \ d\omega(x) = 0 \ \text{ whenever } n \neq m. 
\eeq
By Fubini's Theorem, the products $\prod_{k=1}^2 S^{(n)}(x_k)$ and $\prod_{k=1}^2 S^{(m)}(x_k)$ are then
orthogonal on $R$ with respect to the product measure $\omega(x_1) \times \omega(x_2)$. But clearly,
\beq 
\nu(x_1,x_2) \neq \omega(x_1) \times \omega(x_2) 
\eeq
as $\nu$ cannot be written as the product of a measure in $x_1$ with a measure in $x_2$. This contradicts the
uniqueness of $\nu$ and proves the theorem. 
\end{proof}

\section{Concluding remarks and open questions}
\label{s4}

It is interesting that the asymptotic properties of Stieltjes and Van Vleck polynomials are independent of the fixed charges $\rho_j$, and depend only on the locations $\alpha_j$ of the fixed charges.  Thus, for any configuration of the $\alpha_j$'s, there is a three parameter  $(\rho_1, \rho_2, \rho_3)$ family of Stieltjes polynomials with the same asymptotic properties.  Of course, for finite $k$ the zeros of $S_j^{(k)}$ depend on the $\rho_j$'s, and so our  results are a way to connect the classical results on Stieltjes and Van Vleck polynomials to the more recent work done on the asymptotic properties of these functions.  This allows us to present a fairly complete description in the case we have considered of real $\alpha_j$'s and $p=3$.

 An obvious way to generalize the results we have presented would be to consider the case of more general $p$.  In this case the Van Vleck polynomials are of degree $p-2$.  Hence one question that arises is how to order the Van Vleck and Stieltjes polynomials in a manner analogous to the ordering defined in \eqref{lame2}.  It may be that there is no generally consistent ordering as there is for $p=3$.  However, it still may be possible to generalize some of the results of this paper.

We conclude by noting that an important conjecture in quantum chaos due to M. V. Berry and M. Tabor asserts that for generic quantum integrable systems, the mean level spacing should exhibit a random distribution in the semi-classical limit. It is also believed that a similar behavior holds for the zeros of the corresponding eigenfunctions. An interesting application of our interlacing results may consist of computing the mean level spacing of the quantum integrable systems mentioned in the introduction, and see if the Berry-Tabor conjecture holds in these cases.



\end{document}